\documentclass[10pt]{article}
\usepackage{amsmath, amsfonts, amsthm, amssymb, verbatim}%,dsfont}
\usepackage{graphicx}
\usepackage[mathcal]{euscript}
\usepackage{amstext}
\usepackage{amssymb,mathrsfs}
\usepackage[latin1]{inputenc}
\def\XXint#1#2#3{{\setbox0=\hbox{$#1{#2#3}{\int}$}
\vcenter{\hbox{$#2#3$}}\kern-.5\wd0}}
\DeclareMathOperator\dive {div}
\newcommand{\divg}{\dive_{g}}

 \newcommand{\R}{\mathbb R}

\newcommand{\spt}{\,{\rm spt}}

\newtheorem{theorem}{Theorem}[section]

 \newtheorem{remark}[theorem]{Remark}
\newtheorem{lemma}[theorem]{Lemma}

\newtheorem{definition}[theorem]{Definition}

 \def\beqs{\begin{eqnarray*}}
 \def\enqs{\end{eqnarray*}}
 \def\beq{\begin{eqnarray}}
 \def\enq{\end{eqnarray}}

\begin{document}
\title{A scaling approach to \\
Caffarelli-Kohn-Nirenberg inequality}
%in the Euclidean
%and Riemannian setting}
%
\author{Aldo Bazan$^1$, Wladimir Neves$^1$}

\date{}

\maketitle

\footnotetext[1]{ Instituto de Matem\'atica, Universidade Federal
do Rio de Janeiro, C.P. 68530, Cidade Universit\'aria 21945-970,
Rio de Janeiro, Brazil. E-mail: {\sl aabp2003@pg.im.ufrj.br,
wladimir@im.ufrj.br.}
\newline
\textit{To appear in:}
% \textit{AMS
%Subject Classification.} {Primary: . Secondary: .}
\newline
\textit{Key words and phrases.} Caffarelli-Kohn-Nirenberg
inequality, Sobolev inequality, Hardy inequality, degenerated and
singular elliptic partial differential equations.}

%%%%%%%%%%%%%%%%%%%%%%%%%%%%%%%%%%%%%%%%%%%%%%%%%%%%%%%%%%%%%%%%%
%%%%%%%%%%%%%%%%%%%%%%%%%%%%%%%%%%%%%%%%%%%%%%%%%%%%%%%%%%%%%%%%%
%%%%%%%%%%%%%%%%%%%%%%%%%%%%%%%%%%%%%%%%%%%%%%%%%%%%%%%%%%%%%%%%%
%
%\setlength{\baselineskip}{1.3\baselineskip}

%
%%%%%%%%%%%%%%%%%%%%%%%%%%%%%%%%%%%%%%%%%%%%%%%%%%%%%%%%%%%%%%%%%%%%%%%%%%%%
\begin{abstract}
We consider the general 
Caffarelli-Kohn-Nirenberg
inequality in the Euclidean and Remannian setting. 
From a new parameter introduced, the proof of the former case,
follows by simple interpolation arguments and H\"{o}lder's inequality.
Moreover, the ranges of this convenient parameter completely characterize the inequality. 
Secondly, the same technics are used to study the 
Caffarelli-Kohn-Nirenberg
inequality in the Riemannian case. 
\end{abstract}
%%%%%%%%%%%%%%%%%%%%%%%%%%%%%%%%%%%%%%%%%%%%%%%%%%%%%%%%%%%%%%%%%%%%%%%%%%%%
%

\maketitle

%%%%%%%%%%%%%%%%%%%%%%%%%%%%%%%%%%%%%%%%%%%%%%%%%%%%%%%%%%%%%%%%%%%%%%%%%%%%%%%%%%%%%%%%%%%%%%%%%%%%%%%%%
\section{Introduction} \label{IN}
%%%%%%%%%%%%%%%%%%%%%%%%%%%%%%%%%%%%%%%%%%%%%%%%%%%%%%%%%%%%%%%%%%%%%%%%%%%%%%%%%%%%%%%%%%%%%%%%%%%%%%%%%%

In this paper, we first consider the general form of
Caffarelli-Kohn-Nirenberg inequality  in the Euclidean setting, that is to say,
we study the inequality 
\begin{equation}
  \label{CKN} \Big(\int_{\R^n} \|x\|^{\gamma r}|u|^{r} dx
  \Big)^{1/r}
   \leq C \Big(\int_{\R^n} \|x\|^{\alpha
   p}\left\|\nabla u\right\|^{p} dx \Big)^{a/p}
   \Big(\int_{\R^n} \|x\|^{\beta q}|u|^{q} dx
   \Big)^{(1-a)/q},
\end{equation}
where the real parameters $p$, $q$, $r$, $\alpha$, $\beta$, $\gamma$, satisfy
\begin{equation}
\label{P}
  p,~q \geq 1, \quad r>0 \quad \text{and}
  \quad \gamma r, \; \alpha p, \; \beta q > -n.
\end{equation}
Moreover, a parameter $\sigma$ is introduced by the following convex combination
\begin{equation}
\label{g}
  \forall a \in [\,0,1], \quad \gamma=a \, \sigma + (1-a) \, \beta.
\end{equation}

From a dimensional balance of \eqref{CKN}, we obtain 
\begin{equation}
\label{C} \frac{1}{r}+\frac{\gamma}{n}=a
\left(\frac{1}{p}+\frac{\alpha-1}{n}\right)+(1-a)\left(\frac{1}{q}+\frac{\beta}{n}\right).
\end{equation}
Also, if $a>0$, then $\sigma\leq\alpha$ and, if $a>0$ and
$$\frac{1}{p}+\frac{\alpha-1}{n}=\frac{1}{r}+\frac{\gamma}{n},$$
then $\sigma\geq\alpha-1$. These are necessary and 
sufficient conditions for \eqref{CKN},
as it was proved in \cite{CKN1}.
Here, we be focused on the sufficiency. 
Furthermore, for any compact set in the 
parameter space, such that, \eqref{P}, \eqref{C} and $(\alpha-1) \leq \sigma \leq \alpha$, 
the positive constant $C$ in \eqref{CKN} is bounded. 

\medskip
In this paper the study of Caffarelli-Kohn-Nirenberg inequality relies in a suitable introduced parameter
$s$ as defined in \eqref{S} below, where $0 < s \leq p^*$, and $p^*$ is the Sobolev conjugate of $p$, see \eqref{SOBCONJ}. In fact, 
the proof of  \eqref{CKN} follows by simple interpolation arguments and H\"{o}lder's inequality,
once the new parameter $s$ is considered. Morevoer, we completely characterize the 
inequality from the ranges of $s$. In particular, for $s \in [p,p^*]$ the Caffarelli-Kohn-Nirenberg inequality is proved to be the
interpolation between Hardy's inequality \eqref{HARDYINEQ} and the weighted Sobolev inequality \eqref{WSI}. In this case, the constant  $C>0$, 
which appears in \eqref{CKN} is finite, and it is an exponential convex combination of
the Hardy and the Sobolev constant inequalities.
On the other hand, for $0 < s < p$ the inequality \eqref{CKN} is no more an
interpolation between Hardy and (weighed) Sobolev, further the constant $C$ is not necessarily bounded.   
The characterization of \eqref{CKN} from the ranges of $s$ is new, and we hope clarify the  
understanding of it. 

\medskip
Before we continue to discuss the inequality \eqref{CKN} in the
Euclidean setting, which contains most of the well known
inequalities, we present the Caffarelli-Kohn-Nirenberg inequality
in the Riemannian case, which is also one of the issues of this
paper and a new result, in this generality, from authors'  knowledge. 
Then, we study in Section \ref{REMANNIAN} the following
inequality
\begin{equation}
  \label{CKNREMANNIAN} \Big(\int_{U} \left\|h\right\|^{\gamma r}|u|^{r}dV
  \Big)^{1/r}
   \leq C \Big(\int_{U} \left\|h\right\|^{\alpha
   p}\left\|\nabla u\right\|^{p}dV \Big)^{a/p}
   \Big(\int_{U} \left\|h\right\|^{\beta q}|u|^{q}dV
   \Big)^{(1-a)/q}
\end{equation}
where $U \subset M$ is any open precompact region,
$M$ is a Riemannian n-manifold, with
$n\geq3$,
and $h$ is a special vector field which allows us,
to apply the same technics used before for the Euclidean setting.
The conditions on $M$, which is to say, a complete and non-compact Riemannian
manifold, also with maximal volume growth and non-negative Ricci curvature, will be explained with 
details in Section \ref{REMANNIAN}. Moreover, the special vector field $h$. Those conditions are most 
related to weight's homogeneity, and extra terms on the right hand side of Hardy-Sobolev type 
inequalities on manifolds.  

\medskip
The Caffarelli-Kohn-Nirenberg inequality appeared for the first
time in $\cite{CKN2}$, in that case $p=q=2$ and $a=1$. The paper
$\cite{CKN2}$ introduces the convenient definition of a suitable
weak solution for the incompressible 3D Navier-Stokes equations
with unit viscosity, and the Caffarelli-Kohn-Nirenberg inequality
was used to improve the result established before by Scheffer
concerning the dimension of the subset of singularities. Albeit
\eqref{CKN} appears earlier in the study of incompressible
Navier-Stokes equations, it was soon understood that, this
inequality is important in the theory of elliptic equations, for
instance of the following type
\begin{equation}
\label{e}
  -\dive\big( A(x)\nabla u \big)= f(x,u),
\end{equation}
where $A$ is a nonnegative function that may be unbounded and $f$
is a given function.

\medskip
In different works, the existence and multiplicity of positive or
nodal solutions for \eqref{e} was established, provided the
differential operator
$
  \dive\big(A(x) \nabla (\cdot)\big)
$
is uniformly elliptic (for more details, see \cite{Am} and
\cite{RAB}). Although, interesting and important situations are
obtained in the degenerated and singular cases, respectively
$ \inf A(x)= 0$, $\sup A(x)= \infty$.
For instance, it was studied in $\cite{GHERGU}$ the existence (of
at least two solutions) for the following problem
$$
  -\dive \big( |x|^{-2 s} \, \nabla u \big)
  = K(x)\, |x|^{- \sigma p} \, |u|^{p-2}u + \lambda \; g(x)
  \quad \text{in $\R^{n} \setminus \{0\}$},
$$
where $K\in L^{\infty}(\R^{n})$ (in fact, $K$ has more
conditions), $\lambda$ is a parameter, and $g$ is a continuous
function. The inequality \eqref{CKN} was used to show that the
functional
$$
  J_{\lambda}(u)= \frac{1}{2} \int_{\R^{n}} |x|^{-2 s}
  \|\nabla u \|^{2} \, dx
  - \frac{1}{p} \int_{\R^{n}} K(x) \, |x|^{- \sigma p}|u|^{p} \, dx
   -\lambda\int_{\R^{n}}g(x)u \, dx
$$
is well defined among other properties, that is to say, the
existence of (at least) two critical points for $J_\lambda$.
Similarly, it was studied in \cite{BOU} the existence of
nontrivial solutions for the following problem
$$
  -\dive \big(|x|^{-2 s} \, \nabla u)=  \mu \, |u|^{-2(s+1)} \, u
  + K(x)\, |x|^{2^* \sigma} \, |u|^{2^*-2} u+ \lambda \,  g(x)
  \quad \text{in $\R^{n} \setminus \{0\}$},
$$
where $\mu$ is also a parameter. Again the
Caffarelli-Kohn-Nirenberg inequality \eqref{CKN} was used to show
that the functional
$$
  I_{\lambda,\mu}(u)=\frac{1}{2}\left\|u\right\|^{2}_{\mu}
  -\frac{1}{2^{*}}\int_{\R^{n}} \; K(x)|x|^{-2^{*} \sigma}|u|^{2^{*}}
  -\lambda\int_{\R^{n}}g(x)u
$$
is well defined and the existence of critical points, where
$$
  \|u\|^{2}_{\mu}= \frac{1}{2} \int_{\R^{n}} \Big( |x|^{-2 s}
  \|\nabla u \|^{2} - \frac{\mu}{s} |u|^{-2s} \Big) \, dx.
$$
Therefore, the importance of Caffarelli-Kohn-Nirenberg inequality
\eqref{CKN} is also shown in the two elliptic problems mentioned
before. More information related to applications of this
inequality in elliptic problems can be found in \cite{CIR},
\cite{FKC} and \cite{Pa}. Finally, we highlight that these
singular and degenerate elliptic equations are given models
%for the physical phenomena
(at the equilibrium) for anisotropic media, that are possibly
somewhere between perfect insulators or perfect conductors, see
\cite{LIONS}.

\medskip
Now, let us consider some
particular values of the parameters:

\smallskip
1. ({\bf Sobolev inequality}) 
When $a= 1$, we have by \eqref{g} $\gamma= \sigma$.
Taking $\gamma= \alpha$, it follows by \eqref{C} that, $r= p^*$, and $p< n$, where
\begin{equation}
\label{SOBCONJ}
   p^*:= \frac{n p}{n-p}.
\end{equation}
Then, we obtain from \eqref{CKN} the weighted version of the Sobolev inequality
\begin{equation}
\label{WSI}
     (\int \|x\|^{\alpha p^*} |u(x)|^{p^{*}}dx)^{{1}/{p^{*}}} \leq 
     C_S \; (\int \|x\|^{\alpha p} \left\|\nabla u(x)\right\|^{p}dx)^{{1}/{p}}.
\end{equation}
In particular, for $\alpha= 0$ the weights disappear, 
and we get the usual Sobolev's inequality 
$$(\int|u(x)|^{p^{*}}dx)^{{1}/{p^{*}}} \leq  C_S  \; (\int\left\|\nabla u(x)\right\|^{p}dx)^{{1}/{p}}.$$
Since it was found by Sobolev, many
studies were made to better understand this inequality in
different directions (sharp version, remainder terms, bounded
domains, Riemannian manifolds, etc). More information about this
inequality can be found for example in $\cite{SALOFF-COSTE}$.

\medskip
2.  ({\bf Hardy inequality})  
Again, we take $a= 1$ (thus $\gamma= \sigma$) and consider $\gamma= \alpha - 1$, hence $r= p$. 
Therefore, we obtain from \eqref{CKN} the following version of Hardy's inequality
\begin{equation}
\label{HARDYINEQ}
  (  \int \|x\|^{(\alpha-1) p}  |u(x)|^{p} dx )^{1/p}  \leq C_H \; (\int  \|x\|^{\alpha p} \left\|\nabla u(x)\right\|^{p} dx)^{1/p}.
\end{equation}
In particular, for $\alpha= 0$ we have
$$
    \int \frac{|u(x)|^{p}}{\left\|x\right\|^{p}} dx \leq (C_H)^p \; \int\left\|\nabla u(x)\right\|^{p} dx.
$$
Information about the history of
Hardy inequality can be found in $\cite{KMP}$. This
inequality has also been studied in many different directions (remainder
terms, bounded and unbounded domains, singularity on the boundary,
etc). One interesting application can be found in
$\cite{CHIPOT}$, where this inequality is used to show the
existence of solutions for the Dirichlet problem for the p-Laplace operator in
bounded domains.

\medskip
3. ({\bf Gagliardo-Nirenberg inequality})
When $\alpha=\beta=\sigma=0$ we get $\gamma=0$, and it is
possible to recover from \eqref{CKN} the {Gagliardo-Nirenberg inequality}
$$(\int|u(x)|^{r}dx)^{\frac{1}{r}}\leq C \, (\int\left\|\nabla u(x)\right\|^{p})^{\frac{a}{p}}(\int|u(x)|^{q})^{\frac{1-a}{q}}.$$
In particular, if we consider $a=2/(2+\frac{4}{n})$, $p=2$ and
$q=1$, it follows the important 
{\bf Nash inequality}
$$(\int|u(x)|^{2}dx)^{\frac{1}{2}}\leq C \, (\int\left\|\nabla u(x)\right\|^{2})^{\frac{1}{2+\frac{4}{n}}}(\int|u(x)|)^{\frac{\frac{4}{n}}{2+\frac{4}{n}}}.$$

\medskip
The outline of the paper is the following. First we define a new
parameter that permits us to \textsl{uncouple} the parameter
$\gamma$ and to write the interpolation inequality in an
appropriate way. Then, we use this new parameter to obtain a more
simple inequality that is equivalent to ($\ref{CKN}$) in an
appropriate sense. Next, we give the proof of this inequality in
the Euclidean and the Riemannian case. Finally, we place in the
Appendix some information about the weighted versions of Hardy and Sobolev
inequalities.

%%%%%%%%%%%%%%%%%%%%%%%%%%%%%%%%%%%%%%%%%%%%%%%%%%%%%%%%%%%%%%%%%%%%%%%%%%%%%%%%%%%%
\section{A simplified inequality}
%%%%%%%%%%%%%%%%%%%%%%%%%%%%%%%%%%%%%%%%%%%%%%%%%%%%%%%%%%%%%%%%%%%%%%%%%%%%%%%%%%%%%

In this section, we define the new parameter that, as announced
before at the introduction, allow us to write the original
inequality in a more convenient way.  
From now on, all the
integrals are on $\R^{n}$, and the functions that appear in the
inequalities are test functions, i.e., functions in the space
$C_{c}^{\infty}(\R^{n})$.
\medskip
\newline
First, we rewrite the condition $\eqref{C}$ in the following way
\begin{equation*}
\label{s}
\frac{1}{r}=a\left[\frac{1}{p}-\frac{(\sigma-(\alpha-1))}{n}\right]+\frac{1-a}{q},
\end{equation*}
and define the key parameter $s$ as
\begin{equation}
\label{S} s:=\frac{np}{n-p(\sigma-(\alpha-1))}.
\end{equation}
One observes that $s$ has some useful properties:
\medskip
\newline
a. The parameter $s$ is \textsl{always positive}. Indeed, from \eqref{P} we know that the parameter $r$ is positive, hence if $s$ is supposed negative, from \eqref{s} we would have some values of $a$ for which $r$ would be negative (more precisely, it happens for values of $a$ nearby to $1$). Moreover, if $s$ is zero, then for $a=1$, we would have that $r$ would be infinite, but $r$ is a real number. In both cases, we have a contradiction with the condition that $r$ is a positive real number.
\medskip
\newline
b. As a consequence of item a., we have the following inequality from \eqref{S}
$$n-p(\sigma-(\alpha-1))>0$$
or 
\begin{equation}
\label{np}
n>p(\sigma-(\alpha-1)).
\end{equation}
Therefore, when $\sigma=\alpha$, we have $n>p$. So, it means that this relation between $n$ and $p$ is \textit{implicit} in the inequality 
when $\sigma=\alpha$, and it is important, because (weighted) Sobolev inequality is used exactly in the case $\sigma=\alpha$. 
Observe that, for $\sigma\in[\alpha-1,\alpha)$,
it does not follow necessarily $p<n$. For instance, with $\sigma=\alpha-(1/2)$ we have
$$n>p/2$$
and, if $p=n$, this inequality is still true. More precisely, if we write  
$\sigma= \alpha - \delta$
for any $\delta \in [0,1)$, we must have $p< n / (1-\delta)$. 
On the other hand, for $\sigma \leq (\alpha - 1)$ we do not have any relation between $n$ and 
$p$ that comes from \eqref{np}. However, there exist three relations between $n$ and $p$ (also $\alpha$) that comes from \eqref{P}. The first is
$$\alpha p>-n.$$
The other two relations come from the condition $\gamma r >-n$ in equation \eqref{P}, and equations \eqref{g}, \eqref{C}. They are
$$
\begin{aligned}
(\alpha-1)p&>-n \quad \text{and} \quad
\alpha p^{*}&>-n.
\end{aligned}
$$
These conditions, which are valid for all values of $\sigma$, allow us to use weighted versions of Hardy and Sobolev inequalities.

\medskip
c. The parameter $s$ is an increasing function of $\sigma$. Let $\sigma_{1},\sigma_{2}\in(-\infty,\alpha]$ be such that,
$\sigma_{1 }\leq \sigma_{2}$. Then, we have $\sigma_{1}-\alpha+1 \leq\sigma_{2}-\alpha+1$, and thus
$$
  \frac{np}{n-p(\sigma_{1}-\alpha+1)} \leq \frac{np}{n-p(\sigma_{2}-\alpha+1)},
$$  
that is, $s(\sigma_{1})\leq s(\sigma_{2})$. In particular, if $\sigma\leq \alpha$, then $s\leq p^{*}$, and since $s>0$ we obtain
$$0<s\leq p^{*}.$$ 

\bigskip
To follow, we use the parameter $s$ to define a new inequality. 
The relation between \eqref{CKN} and this new inequality is shown in the following

\begin{lemma}
\label{SIQ}
Assume conditions \eqref{P} and \eqref{g}. If there exist $C>0$ such that
\begin{equation}
\label{UI} \left(\int\left\|x\right\|^{\sigma
s}\left|u(x)\right|^{s}dx\right)^{1/s}\leq
C\left(\int\left\|x\right\|^{\alpha p}\left\|\nabla
u(x)\right\|^{p}dx\right)^{1/p},
\end{equation}
then the inequality $\eqref{CKN}$ holds.
\end{lemma}

\begin{proof}
First, we rewrite the dimensional balance condition
\eqref{C} in the following convenient way
\begin{equation}
   \label{r2}
     r= (1-b) \, q + b \, s,
\end{equation}
where $b \in [0,1]$ is defined as
\begin{equation}
\label{b} b:=\frac{a q}{a q+(1-a)s}.
\end{equation}
Therefore, for any fixed $q$ the parameter $r$ depends on the
values of $s$ and $a$. Now, using \eqref{C} it is possible to get
a convenient inequality, see \eqref{IEQ}, where the importance of
the parameter $s$ becomes more clear. Indeed, we observe first
that for $a=0$ the inequality \eqref{CKN} turns a equality with
$C=1$. Then, we hereupon assume $a>0$ and consequently $b>0$.
Hence from the definition of $\gamma$ and \eqref{C}, we obtain
$$
  \gamma \, r= (1-b)(1-a)\beta q+(1-b)a\sigma q+(1-a)bs\beta+ba\sigma s,
$$
but, we have from \eqref{r2}, \eqref{b} the following
relation between $b$, $a$, $q$ and $s$
\begin{equation}
\label{ri} a(1-b)q =b(1-a)s.
\end{equation}
Consequently, replacing \eqref{ri} in the former equality, it
follows that
$$
  \begin{aligned}
  \gamma \, r&= (1-b)\beta q-(1-b)a\beta q+b(1-a)\sigma s + (1-b) a \beta q + a b \sigma s
\\
  &=(1-b) \beta q - a (1-b) \beta q + b \sigma s - a b \sigma s + a (1-b) \beta q+ a b
  \sigma s
\\
  &=(1-b)\beta q + b \sigma s.
\end{aligned}
$$
Then, we write
$$
  \begin{aligned}
  \Big(\int \left\|x\right\|^{\gamma r}|u|^{r}dx\Big)^{1/r}
&=\Big(\int\left\|x\right\|^{(1-b)\beta q+b \sigma s} \;
\left|u(x)\right|^{(1-b)q+bs}dx \Big)^{1/r}
\\
  &\leq \Big(\int\left\|x\right\|^{\beta
  q} \left|u(x)\right|^{q}dx \Big)^{(1-b)/r}
  \Big(\int\left\|x\right\|^{\sigma s}\left|u(x)\right|^{s}dx \Big)^{b/r},
\end{aligned}
$$
where we have applied Hölder's inequality for $1/(1-b)$ and $1/b$.
Further, using \eqref{r2} and the definition of $b$, we have
\begin{equation}
\label{IEQ}
\begin{aligned}
  \Big(\int \left\|x\right\|^{\gamma r}\left|u(x)\right|^{r}dx \Big)^{1/r}
  \leq
  &\Big(\int\left\|x\right\|^{\beta q}\left|u(x)\right|^{q}dx\Big)^{(1-a)/q}
  \\
  & \times \Big(\int\left\|x\right\|^{\sigma s}\left|u(x)\right|^{s}dx \Big)^{a/s}.
\end{aligned}
\end{equation}
One observes that, the last inequality holds for all admissible
values of the parameters. Moreover, in order to prove \eqref{CKN}
it is enough to show the simpler inequality \eqref{UI}.
\end{proof}

\begin{remark}
From equations \eqref{UI} and \eqref{IEQ} the role of the new parameter
$s$ is clear: We pass from the analysis of the parameter $\gamma$,
that is, in a bidimensional parameter space, to the
analysis of the parameter $s$, that is, in a one
dimensional parameter space. 
\end{remark}

At the end of this section, we state the principal theorem of this paper. 
First, we establish a useful relation between $s,\sigma$ and $n$. Let us recall the condition between $\gamma, r$ and $n$, that is
\begin{equation}
\label{g2}
\gamma r > -n.
\end{equation}
As defined in \eqref{g}, we have that
$$\gamma=a\sigma+(1-a)\beta.$$
On the other hand, using the relation \eqref{C} we obtain
$$r=\frac{sq}{aq+(1-a)s}.$$
Using these equalities in \eqref{g2}, we have for all $a \in [0,1]$,
$$
    a[sq\sigma+nq]+(1-a)[ns+\beta sq]>0.
$$
In particular, for $a=1$, we have
\begin{equation}
\label{ss}
s\sigma>-n.
\end{equation}

\begin{theorem} \label{MAINTHM}
Let $p \geq 1$, $\alpha$, and $\sigma$ be such that $\alpha \, p > -n$,
$\sigma \leq \alpha$. Consider 
$s$ as defined in \eqref{S} satisfying \eqref{ss}. Then, there exists $C>0$, such that \eqref{UI}
holds, that is
$$
\left(\int\left\|x\right\|^{\sigma
s}\left|u(x)\right|^{s}dx\right)^{1/s}\leq
C\left(\int\left\|x\right\|^{\alpha p}\left\|\nabla
u(x)\right\|^{p}dx\right)^{1/p}.
$$
Moreover, when $s \in [p, p^*]$, the constant $C$ is bounded.
\end{theorem}

%%%%%%%%%%%%%%%%%%%%%%%%%%%%%%%%%%%%%%%%%%%%%%%%%%%%%%%%%%%%%%%%%%%%%%%%%%%%%%%%%%%%
\subsection{Proof of Theorem \ref{MAINTHM}. The case $s \in [p,p^*]$}
%%%%%%%%%%%%%%%%%%%%%%%%%%%%%%%%%%%%%%%%%%%%%%%%%%%%%%%%%%%%%%%%%%%%%%%%%%%%%%%%%%%%

The strategy to show the inequality \eqref{UI}, it will be to
interpolate the end-point values of $s$. As observed from the
definition of $s$, this parameter only depends of the value of
$\sigma$, that is, for each value of $\sigma$ we obtain different
values of $s$. In the following table, we summarize the values of
$\sigma$ and the corresponding values of $s$, that we consider
here.

\begin{center}
\begin{tabular}{||l||c|c|r}
\hline
\bfseries $\sigma$ \normalfont & $\alpha-1$ & $\alpha$ \\
\hline
\bfseries $s$ \normalfont & $p$ & $p^{*}$ \\
\hline
\end{tabular}
\end{center}
The proof of \eqref{UI} when $s=p$ and $s=p^{*}$ can be found in
the Appendix. Then, for $s \in (p,p^*)$, first we write $\sigma$
conveniently as
$$
  \sigma=(1-\theta)(\alpha-1)+\theta\alpha,
$$
where $\theta\in[0,1]$ and $s=(1-c)p+cp^{*}$, where $c \in [0,1]$
is given by
$$
  c= \frac{\theta(n-p)}{n-\theta p}.
$$
It follows that
\begin{equation}
\label{s2}
   \sigma s= (1-c)\, (\alpha-1) \, p + \alpha \, c \, p^{*},
\end{equation}
since we have the following relation
$$\theta(1-c)p=(1-\theta)cp^{*}.$$
Now using ($\ref{s2}$), we obtain
$$
  \begin{aligned}
  \int \left\|x\right\|^{\sigma s} \left|u(x)\right|^{s} \, dx&=
  \int \left\|x\right\|^{(1-c)(\alpha-1)p+\alpha
cp^{*}}\left|u(x)\right|^{(1-c)p+cp^{*}}dx
\\[5pt]
  &= \int \Big(\left\|x\right\|^{(\alpha-1)p}|u(x)|^{p} \Big)^{(1-c)}
     \Big(\left\|x\right\|^{\alpha p^{*}}|u(x)|^{p^{*}} \Big)^{c} \, dx
\\[5pt]
  & \leq\Big(\int\left\|x\right\|^{(\alpha-1)p}|u(x)|^{p}dx\Big)^{1-c}
  \Big(\int\left\|x\right\|^{\alpha p^{*}}|u(x)|^{p^{*}}dx\Big)^{c},
  \end{aligned}
$$
where we have applied Hölder's inequality. Therefore, for each $s
\in [p,p^*]$ fixed
$$
  \Big(\int\left\|x\right\|^{\sigma s}\left|u(x)\right|^{s} \, dx \Big)^{1/s}
  \leq C \Big(\int\left\|x\right\|^{\alpha p}\left\|\nabla
  u(x)\right\|^{p} \, dx \Big)^{1/p},
$$
where the constant $C$ is given by
$$
    C= C_{H}^{(1-c) p / s} \; C_{S}^{c p^{*} / s}.
$$

%%%%%%%%%%%%%%%%%%%%%%%%%%%%%%%%%%%%%%%%%%%%%%%%%%%%%%%%%%%%%%%%%%%%%%%%%%%%%%%%%%
\subsection{Proof of Theorem  \ref{MAINTHM}. The case $s \in (0,p)$}
%%%%%%%%%%%%%%%%%%%%%%%%%%%%%%%%%%%%%%%%%%%%%%%%%%%%%%%%%%%%%%%%%%%%%%%%%%%%%%%%%%

Let $0<s<p$ be fixed and define
$$
  \kappa:= \frac{1}{s}-\frac{1}{p},
$$
which is a positive number. Then from the definition of $s$, we
could write
$$
  \sigma =(\alpha-1)- \kappa \, n,
$$
and we have
$$
  \begin{aligned}
   \int \left\|x\right\|^{\sigma s}\left|u(x)\right|^{s} \, dx
   &= \int\left\|x\right\|^{(\alpha-1)s}|u(x)|^{s}\left\|x\right\|^{- \kappa \, n s} \, dx
\\[5pt]
  &\leq \Big( \int  \left\|x\right\|^{(\alpha-1)p}
  |u(x)|^{p} \, dx \Big)^{s/p}
  \; \Big(\int \|x\|^{-n} \, dx \Big)^{(p-s)/p},
  \end{aligned}
$$
where we have used the Hölder inequality with $1/\tilde{p} +
1/\tilde{q}= 1$, for $\tilde{p}= p/s > 1$. 
Denoting $U= \spt(u )$, $R= \sup_{x \in U} \|x\|$, $r= \inf_{x \in U} \|x\|$, we get the following
$$
  \begin{aligned}
  \Big(\int\left\|x\right\|^{\sigma s}\left|u(x)\right|^{s} \, dx \Big)^{1/s}
  &\leq  \Big(\int\left\|x\right\|^{(\alpha-1)p}|u(x)|^{p} \, dx \Big)^{1/p}
  \; \Big( \ln \frac{R}{r} \Big)^{(p-s)/s p} 
\\[5pt]
  &\leq C \, \Big(\int\left\|x\right\|^{\alpha p} \|\nabla u(x)\|^{p} \, dx
  \Big)^{1/p},
  \end{aligned}
$$
where $C= C_H \, \Big(\ln \frac{R}{r} \Big)^\kappa$.

%%%%%%%%%%%%%%%%%%%%%%%%%%%%%%%%%%%%%%%%%%%%%%%%%%%%%%%%%%%%%%%%%%%%%%%%%%%%%%%%%%%%%%%%%%%%%%%%%%%
\section{The Riemannian case} \label{REMANNIAN}
%%%%%%%%%%%%%%%%%%%%%%%%%%%%%%%%%%%%%%%%%%%%%%%%%%%%%%%%%%%%%%%%%%%%%%%%%%%%%%%%%%%%%%%%%%%%%%%%%%%

In this section, we study the Caffarelli-Kohn-Nirenberg 
inequality \eqref{CKNREMANNIAN}, that is, the general 
inequality in the Riemannian setting. 
The proof may follows the same ideas as before, 
hence we just state Lemma \ref{SIQR} and Theorem \ref{MAINTHMR},
which are adapted versions 
from the Euclidean case. In fact, we focus here 
to describe in details, the main
differences which occurs due the inequality \eqref{CKNREMANNIAN}
be posed on complete and non-compact Riemannian
manifolds.

\medskip
First of all, the inequality \eqref{CKN} in \cite{CKN1} was
defined using weight functions of $\|x\|^r$ type, for some $r \in
\R$. The homogeneity of this type of weight functions was one of
the main ingredients in the original proof of the inequality
\eqref{CKN}. Since then, many modifications on the weights, in
particular cases of the \eqref{CKN}, have been considered. For
instance, it was considered in \cite{BT} a cylindrical weight,
i.e., a weight function of the form
$$w(y)=||y||,$$
where $y$ is the projection of a point $x$ in $\R^{n}$ onto
$\R^{n-k}$ (that is $x=(x_{0},y)$, with $x_{0}\in\R^{k}$ and
$y\in\R^{n-k}$). Also, it was considered in \cite{HK} the
following type of weight
$$
%\begin{displaymath}
  w(x)=\left\{
  \begin{aligned}
  \log\Big(\frac{1}{|x|} \Big)\quad &\text{if $n=1$},\\[5pt]
  \log \Big(\frac{R}{||x||} \Big) \quad &\text{if $n \geq 2$},
  \end{aligned}
  \right.
%\end{displaymath}
$$
where $R>1$. Moreover, it can be found in \cite{YZ} an approach of
the inequality \eqref{CKN} with general weights and also a
remainder term. In all these cases, the relationship between the
parameters do not follow the conditions of the original theorem,
because the weights are not necessary homogeneous.
Similar condition happens for weighted inequalities on 
manifolds, from obvious reason. In fact, it does not have a standard way to 
consider weighted inequalities on manifolds. 

\medskip
Concerning the Sobolev inequality in its weighted
version, the distance function is a commonly used weight function,
see \cite{MINERBE, WEI}, but it is not a consensus. In a different
and interesting direction, it was used in \cite{BOZHKOV} the
existence of a conformal Killing vector field $h$ (see the
definition below) on a complete $n$-manifold $M$, $n \geq 3$,  
to prove the following
inequality
$$
  \int_{M} \|h\|^{-p} \; |u|^{p} \, dV \leq
  \Big(\frac{|n-p|}{p}\Big)^{-p} \int_{M}||\nabla u||^{p} \, dV,
$$
where $p>1$, $u\in W^{1,p}(M)$ and $M$ admits a $C^{1}$ conformal
Killing vector field, such that
$ \divg h= n$.
In that paper, more general inequalities were proved, but the case
of the weighted Sobolev inequality was left open. On the other
hand, it was considered in \cite{MINERBE} both inequalities: a
weighted Sobolev and a weighted Hardy, with the distance function
being the weight function. It is given in that paper the proof of
both inequalities under the volume growth assumption, which is not
maximal, but the volume satisfies a doubling condition.

\medskip
Another important difference in the Riemmanian setting 
(not necessarily with weights) from  
the Euclidean case, is concerned an extra term which appears 
on the right hand side, more precisely, let us consider the 
Sobolev inequality.
It can be found in \cite{SALOFF-COSTE}, Theorem 3.3.10, the following version of
the Sobolev inequality: if $M$ is a complete $n$-manifold,
$U\subset M$ is any open precompact region, 
$u\in C_{c}^{\infty}(U)$, and $p\in[1,n)$, then there exists a constant
$C(U,p)$ such that
$$
  \begin{aligned}
  \Big(\int_{U}|u(x)|^{p^{*}}dV\Big)^{1/p^{*}} &\leq
  C(U,p)\Big[\Big(\int_{U}\left\|\nabla u(x)\right\|^{p}dV\Big)^{1/p}
  \\
  &\qquad \qquad
  +\Big(\int_{U}\left|u(x)\right|^{p}dV\Big)^{1/p}\Big].
\end{aligned}
$$
Therefore, we have an extra term which does not appear in the
Euclidean setting. Although, as mentioned in \cite{SALOFF-COSTE}, under
conditions about the volume growth and the Ricci curvature $(Ric  \geq 0)$, 
see Section 3.3.5, applying a pseudo-Poincar\'e type inequality, we obtain Theorem 3.3.11,
where there does not exist 
the second integral on the right hand side of the above
inequality.

\medskip
Finally, it is very important to observe that, as it was showed
in \cite{HEBEY}, it could happen a surprising phenomena in the
Riemannian setting: For any integer $n\geq2$, there exist a
smooth, complete $M$ Riemannian manifold, such that, for each $p
\in [1,n)$, $W^{1,p}(M)$ does not embed in $L^{p^{*}}(M)$, 
see \cite{HEBEY}, Chapter 3, Proposition 3.3. Therefore, 
$M$ be complete, is not a sufficient condition in oder to avoid this 
surprising phenomena.

\medskip
In this section, we consider the inequality \eqref{CKNREMANNIAN} for
functions in $W^{1,p}(M)$, where $M$ is a complete non-compact 
Riemannian $n$-manifold $(n \geq 3)$ with maximal volume growth, $Ric\geq0$, and the weight function is a
conformal Killing vector field. In fact, we assume that the
functions verifying \eqref{CKNREMANNIAN} are in
$C_{c}^{\infty}(U)$, where as before $U\subset M$ is any open precompact region.
The general case, that is $W^{1,p}(U)$ can be obtained by a
standard density argument.
One remarks that, the existence of conformal Killing vector
fields, in the case of a closed manifold (that is, a compact
manifold without boundary) implies that, the Ricci curvature is
non-negative, see \cite{YB}.

\medskip
Now, we recall the definition of a conformal Killing vector
field on a Riemannian manifold.
\begin{definition}
   Let $(M,g)$ be a complete $n$-dimensional Riemannian manifold $(n \geq 3)$, with $g$
   the Riemannian metric. Using local coordinates $(x^{i})_{i=1}^n$,
   we have that $g=(g_{ij})$. A \textbf{nontrivial conformal
Killing vector field} $h=h^{i}\frac{\partial}{\partial x^{i}}$,
   is a vector field on $M$, such that
      $$\nabla^{i}h^{j}+\nabla^{j}h^{i}=
      \frac{2}{n}(\divg h) \, g^{ij} =:\mu \, g^{ij}.$$
\end{definition}
We observe that, $\nabla^{i}(\cdot)$ is the covariant derivative
corresponding to the Levi-Civita connection, which is uniquely
determined by the metric $g$, ($g^{ij}$) is the inverse matrix of
($g_{ij}$), and $\divg h$ is the covariant divergence operator.

\medskip
Following \cite{SALOFF-COSTE}, we define the maximal volume
growth condition of geodesic balls on manifolds. For this, let $V(x,t)$
be the volume of a geodesic ball of radius $t>0$ around a point
$x$ on a manifold $M$. Then, we have the following

\begin{definition}
    Let $(M,g)$ be a complete Riemannian manifold. We say that
    $M$ has a \textbf{maximal volume growth}, if there exist a $c>0$ such that
        $$\forall r>0, \quad V(x,r)\geq cr^{n}.$$
\end{definition}
It is important to remark that, the above definition in
$\cite{SALOFF-COSTE}$ (see page 82), appears associated to the condition that
$Ric\geq0$, i.e., the Ricci curvature is non-negative, such that $M$ satisfies the 
pseudo-Riemannian inequality.

\bigskip
Then, we are in conditon to state the principal results of this section.
First, as in the Euclidean case, using the same parameter
$s$ as defined in \eqref{S}, we have the following 

\begin{lemma}
\label{SIQR}
Let $(M,g)$ be a complete non-compact Riemannnian
$n$-manifold, $n\geq3$, with maximal volume growth, $Ric\geq0$,
and let $U$ be any open precompact region in $M$. 
Assume conditions \eqref{P}, \eqref{g}, 
and $M$ admits a
conformal Killing vector field $h$, with $\divg h= n$. 
If there exist $C>0$ such that
\begin{equation}
\label{UI2R} \left(\int_{U}\left\|h\right\|^{\sigma
s}\left|u(x)\right|^{s}dV\right)^{1/s}\leq
C\left(\int_{U}\left\|h\right\|^{\alpha p}\left\|\nabla
u(x)\right\|^{p}dV\right)^{1/p},
\end{equation}
then the inequality \eqref{CKNREMANNIAN} holds.
\end{lemma}

And hence we pass on the main

\begin{theorem}
\label{MAINTHMR} Under conditions of Lemma \ref{SIQR}, let $p \geq 1$, $\alpha$, and $\sigma$ be such that $\alpha \, p > -n$,
$\sigma \leq \alpha$. Consider 
$s$ as defined in \eqref{S} satisfying \eqref{ss}. Then, there exists $C>0$, such that \eqref{UI2R}
holds. Moreover, when $s \in [p, p^*]$, the constant $C$ is bounded.
\end{theorem}

%%%%%%%%%%%%%%%%%%%%%%%%%%
\section{Appendix} \label{APP}
%%%%%%%%%%%%%%%%%%%%%%%%%%

In this last section, we first state the Hardy type
inequality. The proof follows easily combining the ideas in
\cite{BOZHKOV} and \cite{SALOFF-COSTE}. It is important to note
that, the inequality in the Euclidean case can be
recovered with $h(x)=x$. 

\begin{theorem}
Let $(M,g)$ be a  complete non-compact  Riemannnian $n$-manifold with $n\geq3$,
and $U \subset M$ any open precompact region. If $M$ admits a
conformal Killing vector field, there exists a positive constant
C, such that the following inequality holds for all $u \in
C_{c}^{\infty}(U)$
\begin{equation}
\label{eq1}
\int_{U}\left\|h\right\|^{p(\alpha-1)}|u(x)|^{p}dV \leq C \int_{U}\|h\|^{\alpha p} \|\nabla u(x)\|^p \, dV.
\end{equation}
\end{theorem}

It was stated and proved in $\cite{BOZHKOV}$, the following result
about conformal Killing vector fields: Let $\epsilon$ be an
arbitrary positive real number and $h$ a conformal Killing vector
field, then
$$
  \divg \left(\frac{h}{\epsilon+\|h\|^{k}}\right)=
  \frac{1}{2}\frac{\mu}{(\epsilon+\|h\|^{k})^{2}} \big( n \epsilon + (n-k)\|h\|^{k} \big),
$$
where $k\in\R$. This result is used in that paper to prove a
Caffarelli-Kohn-Nirenberg inequality (particular case) in the Riemannian
setting. In our case, we take $k= (1-\alpha) \, p$ in the above
identity, and the proof of \eqref{eq1} is done using the same
technique that appeared in $\cite{BOZHKOV}$.

\medskip
Now, we state the weighted Sobolev inequality in the Riemannian
setting, and give an original proof of it. One observes that, 
the maximal growth condition is necessary here, since
along the proof we use \textit{standard} Sobolev inequality
(if this condition is not assumed, then the inequality can be
false, see \cite{MINERBE}).

\begin{theorem}
Let $M$ be a complete non-compact Riemannian $n$-manifold with
maximal volume growth, $Ric\geq0$, $n\geq3$ and let $U$ be any
open precompact region in $M$. If $M$ admits a conformal Killing
vector field, then there exists a positive constant C, such that
the following inequality holds for all $u \in C_{c}^{\infty}(U)$
$$
  \Big(\int_{U}\left\|h\right\|^{\alpha p^{*}}|u(x)|^{p^{*}} \, dV \Big)^{1/p^{*}} \, dV
  \leq C \Big(\int_{U}\left\|h\right\|^{\alpha p}\left\|\nabla u(x)\right\|^{p} \, dV
  \Big)^{1/p}.
$$
\end{theorem}

\begin{proof} 1. First, we consider the following result: 

Claim:  If $h$ is a vector field on $M$,
  then for almost all $x \in M$
   \begin{equation}
   \label{eq2}
   h(x) \cdot \nabla(\left\|h(x)\right\|)=\left\|h(x)\right\|,
   \end{equation}
  where the inner product is taking with respect to the
  Riemannian metric $(g_{ij})$. 
   
  Proof of Claim: The proof of \eqref{eq2} follows the ideas in
  Cordero-Eurasquim, Nazaret, Villani \cite{VILLANI}.
  Given $\epsilon>0$, we define a function $\lambda$, such that
   \begin{displaymath}
   \begin{aligned}
   \lambda:(1-\epsilon,1+\epsilon)\longrightarrow &X(M)\\
   t~~~~~~~~\longmapsto &t \, h,
   \end{aligned}
   \end{displaymath}
  where $X(M)$ is the space of vector fields on $M$.
  Observe that, $\lambda$ is differentiable in ($1-\epsilon,1+\epsilon$),
  and $\lambda'(t)= h$. To follow, we define the function $\eta$ as
    $$\eta(t):= \| \lambda(t) \|.$$
  Then, we have for almost all $t$,
    $$\eta'(t)=\nabla\left\|h\right\| \cdot h$$
  On the other hand, using the definition of $\eta$, we obtain
  $\eta'(t)=\left\|h\right\|$. Thus, we obtain
  $$
    h \cdot \nabla\left\|h\right\|=\left\|h\right\|.
  $$

\medskip
2. For any $g \in C_{c}^{\infty}(U)$, we have for each 
open precompact region $U \subset M$
\begin{equation}
\label{SSI}
    \left(\int_{U}|g|^{p^{*}}dV\right)^{1/p^{*}}\leq C\left(\int_{U}\left\|\nabla g \right\|_{*}^{p}dV\right)^{1/p},
\end{equation}
where $C$ is a positive constant and $\left\|\cdot\right\|_{*}$ denotes
the dual norm of a vector field. Also, we 
recall the following simple inequality of real numbers: 
If $a,b\in\R$ and $k\geq1$,
then
\begin{equation}
\label{eq3}
(a+b)^{k}\leq2^{k-1}(a^{k}+b^{k}).
\end{equation}
Now, defining $f(x):= \left\|h\right\|^{\alpha}u(x)$, we have
\begin{equation}
\label{S1}
  \int_{U}|f(x)|^{p^{*}}dV=\int_{U}\left\|h\right\|^{\alpha p^{*}}|u(x)|^{p^{*}} \, dV.
\end{equation}
On the other hand
 \begin{equation}
 \label{GF}
 \nabla f(x)=
 %\nabla(\left\|h\right\|^{\alpha}u(x))
 \alpha\left\|h\right\|^{\alpha-1}u(x) \, \nabla(\left\|h\right\|)+\left\|h\right\|^{\alpha}\nabla u(x),
 \end{equation}
and making the inner product between $h$ and $\nabla f(x)$, we
have
$$
   h \cdot \nabla f(x)=\alpha\left\|h\right\|^{\alpha-1}u(x) \, h
   \cdot \nabla\left\|h\right\|+\left\|h\right\|^{\alpha} h \cdot \nabla u(x).
$$
From the above equality and  \eqref{eq2}, it follows that $(h \neq 0$, $h= 0$ is trivial)
 \begin{equation}
 \label{HH}
  \tilde{h} \cdot \nabla f(x) \leq  \alpha\left\|h\right\|^{\alpha-1} \, |u(x)| 
   +\left\|h\right\|^{\alpha} \, \|\nabla u(x)\|,
 \end{equation}
 where $\tilde{h}= h / \|h\|$. Replacing $h$ by $-h$, we do not change the defining of $f$. 
Furthermore, it is not difficult to see that, we also have the estimate
 \begin{equation}
 \label{HH1}
 - \tilde{h} \cdot \nabla f(x) \leq  \alpha\left\|h\right\|^{\alpha-1} \, |u(x)| 
   +\left\|h\right\|^{\alpha} \, \|\nabla u(x)\|.
 \end{equation}
 Moreover, from \eqref{GF} we observe that, there exists $K_0 > 0$, such that
 $$K_0 \leq |\cos(\tilde{h}, \nabla f)|.$$
Consequently, we obtain from \eqref{HH}, \eqref{HH1} 
and the definition of the dual norm
$$
K_0 \,  \left\|\nabla f(x)\right\|_* \leq\alpha\left\|h\right\|^{\alpha-1}|u(x)|
  +\left\|h\right\|^{\alpha}\left\|\nabla u(x)\right\|.
$$
Hence applying \eqref{eq3}, it follows that
\begin{equation}
\label{HHH}
  \left\|\nabla f(x)\right\|_{*}^{p}\leq K_{1}\left\|h\right\|^{(\alpha-1)p}|u(x)|^{p}
  + K_{2}\left\|h\right\|^{\alpha p}|\nabla u(x)|^{p},
\end{equation}
where $K_{1}$ and $K_{2}$ are positive constants. 

\medskip
3. Finally, we integrate \eqref{HHH} on $U$ to obtain
$$
\begin{aligned}
 \Big(\int_{U}\left\|\nabla f\right\|_{*}^{p}dV\Big)^{p^{*}/p}
 \leq &K_{3}  \Big(\int_U \left\|h\right\|^{(\alpha-1)p}|u(x)|^{p}\Big)^{p^{*}/p}
 \\[5pt]
 + &K_{4}\Big(\int_U \left\|h\right\|^{\alpha p}|\nabla u(x)|^{p}\Big)^{p^{*}/p}.
\end{aligned}
$$
In the first integral of the right hand side of the above inequality, we apply the
weighted Hardy inequality \eqref{eq1}, then
\begin{equation}
\label{S2}
 \Big(\int_{U}\left\|\nabla f\right\|_{*}^{p}dV\Big)^{p^{*}/p}\leq K_{5}\Big(\int_U \left\|h\right\|^{\alpha p}|\nabla u(x)|^{p}\Big)^{p^{*}/p}.
\end{equation}
From \eqref{SSI}, \eqref{S1}, and \eqref{S2}, we show the thesis of the theorem, that is
$$\Big(\int_{U}\left\|h\right\|^{\alpha p^{*}}|u(x)|^{p^{*}} \, dV \Big)^{1/p^{*}} \, dV\leq C \, 
\Big(\int_{U}\left\|h\right\|^{\alpha p}\left\|\nabla u(x)\right\|^{p} \, dV \Big)^{1/p}.$$
\end{proof}

%%%%%%%%%%%%%%%%%%%%%%%%%%%%%%%%%%%%%%%%%%%%%%%%%%%%%%%%%%%%%%%%%%%%%%%%%%%%%%%%%%%%%%%%%%%%%%%%%%%%%%%%%%%%%%%%%
\section*{Acknowledgements}
%%%%%%%%%%%%%%%%%%%%%%%%%%%%%%%%%%%%%%%%%%%%%%%%%%%%%%%%%%%%%%%%%%%%%%%%%%%%%%%%%%%%%%%%%%%%%%%%%%%%%%%%%%%%%%%%%

Wladimir Neves is partially supported by CNPq
through the grants 484529/2013-7, 308652/2013-4.

\end{document}